\documentclass[12pt]{article}

\usepackage{amsmath, amsthm, amssymb, enumerate}
\usepackage{color}
\usepackage{datetime}
\usepackage{fancyhdr}
\usepackage{algorithm}
\usepackage{algpseudocode}
\usepackage{pifont}

\chardef\coloryes=1 
\chardef\isitdraft=1 

\ifnum\isitdraft=1 
   \def\version{9} 
   \def\eqref#1{({\ref{#1}})}                

\usepackage[color,notcite]{showkeys}

\usepackage[notcite]{showkeys}

\definecolor{labelkey}{gray}{.3}

\definecolor{refkey}{rgb}{.3,0.3,0.3}

\else

  \def\startnewsection#1#2{\section{#1}\label{#2}\setcounter{equation}{0}}   
  \def\nnewpage{} 
\fi

\begin{document}
\def\ques{{\cor \underline{??????}\cob}}
\def\nto#1{{\coC \footnote{\em \coC #1}}}
\def\fractext#1#2{{#1}/{#2}}
\def\fracsm#1#2{{\textstyle{\frac{#1}{#2}}}}   
\def\nnonumber{}


\def\cor{{}}
\def\cog{{}}
\def\cob{{}}
\def\coe{{}}
\def\coA{{}}
\def\coB{{}}
\def\coC{{}}
\def\coD{{}}
\def\coE{{}}
\def\coF{{}}

\ifnum\coloryes=1

  \definecolor{coloraaaa}{rgb}{0.1,0.2,0.8}
  \definecolor{colorbbbb}{rgb}{0.1,0.7,0.1}
  \definecolor{colorcccc}{rgb}{0.8,0.3,0.9}
  \definecolor{colordddd}{rgb}{0.0,.5,0.0}
  \definecolor{coloreeee}{rgb}{0.8,0.3,0.9}
  \definecolor{colorffff}{rgb}{0.8,0.3,0.9}
  \definecolor{colorgggg}{rgb}{0.5,0.0,0.4}

 \def\cog{\color{colordddd}}
 \def\cob{\color{black}}
 \def\cor{\color{red}}
 \def\coe{\color{colorgggg}}

 \def\coA{\color{coloraaaa}}
 \def\coB{\color{colorbbbb}}
 \def\coC{\color{colorcccc}}
 \def\coD{\color{colordddd}}
 \def\coE{\color{coloreeee}}
 \def\coF{\color{colorffff}}
 \def\coG{\color{colorgggg}}

\fi
\ifnum\isitdraft=1
   \chardef\coloryes=1 
   \baselineskip=17pt
\pagestyle{myheadings}
\reversemarginpar

\def\const{\mathop{\rm const}\nolimits}  
\def\diam{\mathop{\rm diam}\nolimits}    

 \def\llabel#1{\label{#1}{\ \mbox{\rm\color{red} {#1}\color{black}}}}

\def\rref#1{{\ref{#1}{\rm \tiny \fbox{\tiny #1}}}}
\def\theequation{\fbox{\bf \thesection.\arabic{equation}}}
\def\ccite#1{{\cite{#1}{\rm \tiny ({#1})}}}

\def\startnewsection#1#2{\newpage\cog \section{#1}\cob\label{#2}

\setcounter{equation}{0}
\pagestyle{fancy}

\lhead{\cob Section~\ref{#2}, #1 }
\cfoot{}
\rfoot{\thepage}
\lfoot{\cob{\today,~\currenttime}~(c75-iklt2, Version~\fbox{\version})}}
\chead{}
\rhead{\thepage}
\def\nnewpage{\newpage}

\newcounter{startcurrpage}
\newcounter{currpage}

\def\llll#1{{\rm\tiny\fbox{#1}}}
   \def\blackdot{{\color{red}{\hskip-.0truecm\rule[-1mm]{4mm}{4mm}\hskip.2truecm}}\hskip-.3truecm}
   \def\bdot{{\coC {\hskip-.0truecm\rule[-1mm]{4mm}{4mm}\hskip.2truecm}}\hskip-.3truecm}
   \def\purpledot{{\coA{\rule[0mm]{4mm}{4mm}}\cob}}
   \def\pdot{\purpledot}
\else
   \baselineskip=15pt
   \def\blackdot{{\rule[-3mm]{8mm}{8mm}}}
   \def\purpledot{{\rule[-3mm]{8mm}{8mm}}}
   \def\pdot{}
\fi

\def\tdot{\fbox{\fbox{\bf\tiny I'm here; \today \ \currenttime}}}
\def\nts#1{{\hbox{\bf ~#1~}}} 
\def\nts#1{{\cor\hbox{\bf ~#1~}}} 
\def\ntsf#1{\footnote{\hbox{\bf ~#1~}}} 
\def\ntsf#1{\footnote{\cor\hbox{\bf ~#1~}}} 
\def\bigline#1{~\\\hskip2truecm~~~~{#1}{#1}{#1}{#1}{#1}{#1}{#1}{#1}{#1}{#1}{#1}{#1}{#1}{#1}{#1}{#1}{#1}{#1}{#1}{#1}{#1}\\}
\def\biglineb{\bigline{$\downarrow\,$ $\downarrow\,$}}
\def\biglinem{\bigline{---}}
\def\biglinee{\bigline{$\uparrow\,$ $\uparrow\,$}}

\def\tilde{\widetilde}

\newtheorem{Theorem}{Theorem}[section]
\newtheorem{Corollary}[Theorem]{Corollary}
\newtheorem{Proposition}[Theorem]{Proposition}
\newtheorem{Lemma}[Theorem]{Lemma}
\newtheorem{Remark}[Theorem]{Remark}
\newtheorem{Example}[Theorem]{Example}
\newtheorem{Assumption}[Theorem]{Assumption}
\newtheorem{definition}{Definition}[section]
\def\theequation{\thesection.\arabic{equation}}
\def\endproof{\hfill$\Box$\\}
\def\square{\hfill$\Box$\\}
\def\comma{ {\rm ,\qquad{}} }            
\def\commaone{ {\rm ,\qquad{}} }         
\def\dist{\mathop{\rm dist}\nolimits}    
\def\sgn{\mathop{\rm sgn\,}\nolimits}    
\def\Tr{\mathop{\rm Tr}\nolimits}    
\def\div{\mathop{\rm div}\nolimits}    
\def\supp{\mathop{\rm supp}\nolimits}    
\def\divtwo{\mathop{{\rm div}_2\,}\nolimits}    
\def\re{\mathop{\rm {\mathbb R}e}\nolimits}    
\def\div{\mathop{\rm{Lip}}\nolimits}   
\def\indeq{\qquad{}}                     
\def\period{.}                           
\def\semicolon{\,;}                      

\title{Controlled Markov Processes with AVaR Criteria for Unbounded Costs}
\author{Kerem U\u{g}urlu}
\maketitle

\date{}

\begin{center}
\end{center}

\medskip

\indent Department of Mathematics, University of Southern California, Los Angeles, CA 90089\\
\indent e-mails:kugurlu@usc.edu

\begin{abstract}
In this paper, we consider the control problem with the Average-Value-at-Risk (AVaR) criteria of the possibly unbounded $L^{1}$-costs in infinite horizon on a Markov Decision Process (MDP). With a suitable state aggregation and by choosing a priori a global variable $s$ heuristically, we show that there exist optimal policies for the infinite horizon problem.
\end{abstract}

\noindent\thanks{\em Mathematics Subject Classification\/}: 90C39, 93E20

\noindent\thanks{\em Keywords:\/}Markov Decision Problem, Average-Value-at-Risk, Optimal Control;

\section{Introduction} 

In classical models, the optimization problem has been solved by expected performance criteria. Beginning with Bellman \cite{ref3}, risk neutral performance evaluation has been used via dynamic programming techniques. This methodology has seen huge development both in theory and practice since then. However, in practice expected values are not appropriate  to measure the performance criteria. Due to that, risk aversive approaches have been begun to forecast the corresponding problem and its outcomes specifically by utility functions (see e.g. \cite{key-200,key-21}).  To put risk-averse preferences into an axiomatic framework, with the seminal paper of Artzner et al. \cite{key-1}, the risk assessment gained new aspects for random outcomes. In \cite{key-1}, the concept of \textit{risk measure} has been defined and theoretical framework has been established. We will use this framework to measure risk aversion. We replace the risk neutral expectation operator with this risk averse operator and study the optimal control of infinite sum of cost functions and characterize the optimal policy stationary as in risk neutral case but in a \textit{state-aggregated} setting. 

The rest of the paper is as follows. In Section 2, we give the preliminary theoretical framework. In Section 3, we derive the dynamic programming equations for MDP using AVaR criteria for the infinite time horizon and conclude the paper by giving an application of our results to classical LQ problem to illustrate our results.
\subsection{Controlled Markov Processes}
We take the control model $\mathcal{M} = \{ \mathcal{M}_n, n \in \mathbb{N}_0 \}$, where for each $n \in \mathbb{N}_0$, 
\begin{equation}
\mathcal{M}_n := (X_n, A_n, \mathbb{K}_n, F_n, c_n)
\end{equation}
with the following components:
\begin{itemize}
\item $X_n$ and $A_n$ denote the state and action (or control) spaces, where $X_n$ take values in a Borel set $X$ whereas $A_n$ take values in a Borel set $A$. 
\item For each $x \in X_n$, let $A_n(x) \subset A_n$ be the set of all admissible controls in the state $x_n = x$.  Then 
\begin{equation}
\mathbb{K}_n := \{ (x,a): x \in X_n, a \in A_n(x) \},
\end{equation}
stands for the set of feasible state-action pairs at time $n$, where we assume that $\mathbb{K}_n$ is a Borel subset of $X_n \times A_n$.
\item We let $x_{n+1} = F_n(x_n,a_n, \xi_n)$, for all $n = 0,1,...$ with $x_n \in X_n$ and $a_n \in A_n$ as described above, with independent random disturbances $\xi_n \in S_n$ having probability distributions $\mu_n$, where the $S_n$ are Borel spaces. 
\item $c_n(x,a): \mathbb{K}_n \rightarrow \mathbb{R}$ stands for the deterministic cost function at stage $n \in \mathbb{N}_0$ with  $(x,a) \in \mathbb{K}_n$.
\end{itemize}
The random variables $\{\xi_n\}_{n\geq 0}$  are defined on a common probability space $(\Omega, \mathcal{F},\{\mathcal{F}_n\}_{n\geq 0}, \mathbb{P})$, where $\mathbb{P}$ is the reference probability space with each $\xi_n$ measurable with respect to sigma algebra $\mathcal{F}_n$ with $\mathcal{F} = \sigma (\cup_{n=0}^\infty \mathcal{F}_n)$. Based on the action $a \in \mathbb{K}_n(x)$ chosen at time $n$, we assume that $A_n$ is $\mathcal{F}_n = \sigma(X_0,A_0,...,X_n)$-measurable, i.e. our decision might depend entirely on the history $h_n$, where $h_n = (x_0,a_0,x_1,...,a_{n-1}, x_n) \in H_n$ is the history up to time $n$, where define recursively 
\begin{equation}
H_0 := X, \quad H_{n+1} := H_n \times A \times X
\end{equation} For each $n \in \mathbb{N}_0$, let $\mathbb{F}_n$ be the family of measurable functions $f_n: H_n \rightarrow A_n$ such that 
\begin{equation}
f_n (x) \in A_n(x),
\end{equation}
for all $x \in X_n$. A sequence $\pi = \{ f_n \}$ of functions $f_n \in \mathbb{F}_n$ for all $n \in \mathbb{N}_0$ is called a policy. We denote by $\Pi$ the set of all the policies. Then for each policy $\pi \in \Pi$ and initial state $x \in \mathbf{X}$, a stochastic process $\{ (x_n,a_n) \}$ and a probability measure $\mathbb{P}_x^\pi$ is defined on $(\Omega,\mathcal{F})$ in a canonical way, where $x_n$ and $a_n$ represent the state and the control at time $n \in \mathbb{N}_0$. The expectation operator with respect to $\mathbb{P}_x^\pi$ is denoted by $\mathbb{E}_x^\pi$.The distribution of $X_{n+1}$ is given by the transition kernel $\mathbb{Q}$ from $X \times A$ to $X$  as follows:
\begin{align*}
&P^{\pi}(X_{n+1} \in B_x | X_0, g_0(X_0),...,X_n,f_n(X_0,A_0,...,X_n))
\\&=P^{\pi}(X_{n+1} \in B_x | X_n,f_n(X_0,A_0,...,X_n) )
\\&=\mathbb{Q}(B_x | X_n, g_n(X_0,A_0,...,X_n))
\end{align*}
for Borel measurable sets $B_x \subset \mathbf{X}$.
A Markov policy is of the form 
\begin{equation}
P^{\pi}(X_{n+1} \in B_x| X_n,f_n(X_0,A_0,...,X_n))
= \mathbb{Q}(B_x | X_n, f_n(X_n)).
\end{equation} That is to say, the Markov policy $\pi = \{f_n\}_{n \geq 0}$ depends only on current state $X_n$. We denote the set of all Markovian policies as $\Pi^M$. Similarly, the stationary policy is of the form $\pi = \{f\}_{n \geq 1}$ with 
\begin{equation}
P^{\pi}(X_{n+1} \in B_x| X_n,f_n(X_0,A_0,...,X_n))
= \mathbb{Q}(B_x | X_n, f(X_n)),
\end{equation}
i.e. we apply the same rule for each time episode $n$.
Suppose, we are given a policy $\sigma=\{f_{n}\}_{n=0}^{\infty}$, then by Ionescu Tulcea theorem \cite{key-6}, there exists a unique probability measure $P^{\sigma}$ on $(\Omega,\mathcal{F})$, which ensures the consistency of the infinite horizon problem considered. Hence, for every measurable set $B\subset\mathcal{F}_n$ and all $h_n \in H_n$, $n \in \mathbb{N}_0$, we denote
\begin{align*}
P^{\sigma}(x_{1}\in B) =& P(B) \\
P^{\sigma}(x_{n+1}\in B|h_{n}) =& Q(B|x_{n},\pi_{n}(h_{n}))
\end{align*}
We consider the following cost function
\begin{equation}
C^{\infty}:=\sum_{n=0}^{\infty} c_{n}(x_n,a_n),
\end{equation} for the infinite planning horizon and 
\begin{equation}
C^N = \sum_{n=0}^N c_n(x_n,a_n)
\end{equation} for the finite planning horizon for some terminal time $N \in \mathbb{N}_0$. We take that the cost functions $\{c_{n}(x_n,a_n)\}_{n \geq 0}$ are non-negative and $C^N$ and $C^{\infty}$ belong to
space $L^{1}(\Omega,\mathcal{F},\mathbb{P}_0)$. We start from the following two well-studied optimization problems for controlled Markov processes. The first one is called \textit{finite horizon expected value problem},
where we want to find a policy $\pi=\{g_n\}_{n=0}^{N}$ with
the minimization of the expected cost:
\begin{align*}
\min_{\pi \in \Pi}\mathbb{E}_x^{\pi}[\sum_{n=0}^{N}c_{n}(x_{n},a_{n})]
\end{align*}
where $a_{n}=\pi_{n}(x_{0},x_{1},...,x_{n})$ and $c_{n}(x_{n},a_{n})$
is measurable for each $n=0,..,N$. The second problem is the infinite horizon expected value problem. The objective is to find a policy $\pi=\{g_{n}\}_{n=0}^{\infty}$ with
the minimization of the expected cost:
\begin{align*}
\min_{\pi \in \Pi}\mathbb{E}_x^{\pi}[\sum_{n=0}^{\infty} c_{n}(x_{n},a_{n})]
\end{align*}
Under some assumptions the first optimization problem has solution in form of Markov policies, whereas in infinite case the policy is stationary. In both cases, the optimal policies can be found by solving corresponding dynamic programming equations. Our goal is to study the infinite horizon problem, where we use a \textit{risk-averse operator} $\rho$ instead of the expectation operator and look for stationary optimal policy under some conditions.
\subsection{Coherent risk measures on $L^1$}
We introduce the corresponding risk averse operators that
we will be working on throughout the rest of the paper. 
\begin{definition}
A function $\rho:L^1\rightarrow\mathbb{R}$ is said to
be a \textit{coherent risk measure} if it satisfies the following
axioms \cite{key-1}:
\begin{itemize}
\item $\rho(\lambda X+(1-\lambda)Y)\leq\lambda\rho(X)+(1-\lambda)\rho(Y)$
$\forall\lambda\in(0,1)$, $X,Y \in L^p$ \label{convexity-1};
\item If $X \leq Y$ $\mathbb{P}-$a.s. then $\rho(X) \leq \rho(Y)$, $\forall X,Y \in L^p$
\item $\rho(c+X) = c + \rho(X)$, $\forall c\in\mathbb{R}$, $X\in L^{p}$;
\item $\rho(\beta X)=\beta\rho(X),$ $\forall X\in L^p$, $\beta \geq 0$.
\end{itemize}
\end{definition}
The particular risk averse operator that we will be working with is
the $\mathrm{AVaR}_{\alpha}(X)$
\begin{definition}
Let $X\in L^1(\Omega,\mathcal{F},\mathbb{P})$ be a real-valued
random variable and let\\ $\alpha \in(0,1)$. 
\end{definition}
\begin{itemize}
\item We define the \textit{Value-at-Risk} of $X$ at level $\alpha$, $\mathrm{VaR}_{\alpha}(X)$, by 
\begin{equation}
\mathrm{VaR}_{\alpha}(X) = \inf \left \{x\in\mathbb{R}:\mathbb{P}(X\leq x)\geq \alpha \right \}
\end{equation}
\item We define the coherent risk measure, the \textit{Average-Value-at-Risk}
of $X$ at level $\alpha$, denoted by $\mathrm{AVaR}_{\alpha}(X)$ as 
\begin{equation}
\mathrm{AVaR}_{\alpha}(X)=\frac{1}{1-\alpha}\int_{\alpha}^{1}\mathrm{VaR}_t(X)dt
\end{equation}
\end{itemize}
We will also need the following alternative representation for $\mathrm{AVaR}_{\alpha}(X)$
as shown in \cite{key-8}.
\begin{Lemma}
\label{lemma_static_avar_repre:}Let $X\in L^{p}(\Omega,\mathcal{F},\mathbb{P})$
be a real-valued random variable and let 
$\alpha \in(0,1)$. Then it holds that 
\begin{equation} 
\mathrm{AVaR}_{\alpha}(X)=\min_{s\in\mathbb{R}}\left\{s+\frac{1}{1-\alpha}\mathbb{E}[(X-s)^{+}] \right \},\label{avar_representation}
\end{equation}
where the minimum is attained at $s = \mathrm{VaR}_{\alpha}(X)$.
\end{Lemma}
\begin{Remark}
We note from the representation above that the $\mathrm{AVaR}_\alpha(X)$ is real-valued for any $X \in L^1(\Omega,\mathcal{F},\mathbb{P})$. 
\end{Remark}
\subsection{Time Consistency}
\begin{definition}
Let $L^0(\mathcal{F}_n)$ be the vector space of all real-valued, $\mathcal{F}_n$-measurable random variables on the space $(\Omega, \mathcal{F},\mathcal{F}_n, \mathbb{P})$ defined above. A one-step coherent dynamic risk measure on $L^0(\mathcal{F}_{n+1})$ is a sequence of mappings such that 
\begin{equation}
\rho_t : L^0(\mathcal{F}_{n+1}) \rightarrow L^0(\mathcal{F}_n), n = 0,..., N-1. 
\end{equation}
that satify the followings
\begin{itemize}
\item $\rho_n(\lambda X+(1-\lambda)Y)\leq\lambda\rho_n(X)+(1-\lambda)\rho_n(Y)$
$\forall\lambda\in(0,1)$, $Z,W \in L^0(\mathcal{F}_{n+1})$ \label{convexity-1};
\item If $X \leq Y$ $\mathbb{P}-$a.s. then $\rho_{n+1}(X)\leq \rho_{n+1}(Y)$, $\forall X,Y\in L^0(\mathcal{F}_{n+1})$
\item $\rho_n(c+X)=c+\rho_n(X)$, $\forall c \in L^0(\mathcal{F}_{n})$, $X\in L^0(\mathcal{F}_{n+1})$;
\item $\rho_n(\beta X)=\beta\rho_n(X),$ $\forall X\in L^0(\mathcal{F}_{n})$, $\beta \geq 0$
\end{itemize}
\end{definition}
\begin{definition} A dynamic risk measure $(\rho_n)_{n=0}^{N-1}$ on $L^0(\mathcal{F}_N)$ is called time-consistent if for all $X,Y \in L^0(\mathcal{F}_N)$ and $n = 0,..., N-1$, $\rho_{n+1}(X) \geq \rho_{n+1}(Y)$ implies $\rho_n(X) \geq \rho_n(Y)$. 
\end{definition}

Another way to define time consistency is from the point of view of optimal policies (see also \cite{key-15}). Intuitively, the sequence of optimization problems is said to be dynamically consistent, if the optimal strategies obtained when solving the original problem at time $t$ remain optimal for all subsequent problems. More precisely, if a policy $\pi$ is optimal on the time interval $[s,T]$, then it is also optimal on the sub-interval $[t,T]$ for every $t$ with $s\leq t \leq T$. 

\begin{Remark} \label{Remark31} Given that the probability space is atomless, it is shown in \cite{key-15} and \cite{key-18} that the only law invariant coherent risk measure operators $\rho$ on $(\Omega, \mathcal{F}, \mathcal{F}_{n=0}^N, \mathbb{P})$, i.e. 
\begin{equation}
X \stackrel{d}= Y \Rightarrow \rho(X) = \rho(Y)
\end{equation} satisfying
\begin{equation}
\rho (Z) = \rho( \rho| \mathcal{F}_1(... \rho| \mathcal{F}_{N-1})(Z) ),
\end{equation}
for all random variables $Z$ are $\mathrm{esssup}(Z)$ and expectation $\mathbb{E}(Z)$ operators. This suggests that optimization problems with most of the coherent risk measures are not time consistent.  
\end{Remark}
\section{Infinite Horizon Problem}
We are interested in solving the following optimization problem in the infinite horizon. 
\begin{equation}
\label{main_problem}
\min_{\pi \in \Pi}\mathrm{AVaR}_{\alpha}^\pi(\sum_{n=0}^{\infty} c(x_n,a_n)),
\end{equation}
First, we put the following assumptions on the problem.
\begin{Assumption} \label{assumptions}
For every $n \in \mathbb{N}_0$, we impose the following assumptions  on the problem:
\begin{enumerate} 
\item The cost function $c_n: \mathbb{K}_n \rightarrow \mathbb{R}$ is nonnegative, lower semicontinuous (l.s.c.), that is if $(x_k,a_k) \rightarrow (x,a)$, then 
\begin{equation}
\liminf_{k \rightarrow \infty} c_n(x^k, a^k) \geq c_n (x,a)
\end{equation} and inf-compact on $\mathbb{K}_n$, i.e., for every $x \in X_n$ and $r \in \mathbb{R}$, the set 
\begin{equation}
\{ a \in A_n(x)| c_n(x,a) \leq r \}
\end{equation}
is compact.
\item The function $(x,a) \rightarrow \int v(x',s-c)\mathbb{Q}(dx'| x,a)$ is l.s.c. for all l.s.c. functions $v \geq 0$. 
\item The set $A_n(x)$ is compact for every $x \in X_n$ and for every $n \in \mathbb{N}_0$.
\item The system function $x_{n+1} = F_n(x_n,a_n,\xi_n)$ is measurable as a mapping\\ $F_n: X_n \times A_n \times S_n \rightarrow X_{n+1}$, and $(x,a) \rightarrow F_n(x,a,s)$ is continuous on $\mathbb{K}_n$ for every $s \in S_n$.
\item The multifunction also called point-to-set function $x \rightarrow A_n(x)$, from $\mathbf{X}$ to $A$ is upper semicontinuous (u.s.c.) that is, if $\{x_n\} \subset \mathbb{X}$ and $\{a_n\} \subset A$ are sequences such that 
\begin{equation}
x_n \rightarrow x^*, \quad a_n \in A(x_n) \quad \mathrm{ for all } \quad n, \quad \mathrm{ and }\quad a_n \rightarrow a^*, 
\end{equation}
then $a^*$ is in $A_n(x^*)$.
\item There exists a policy $\pi \in \Pi$ such that $V_0(x,\pi) < M$ for all $x \in X_0$. 
\end{enumerate}
\end{Assumption}
To solve (\ref{main_problem}), we first rewrite the infinite horizon problem as follows:
\begin{align*}
\inf_{\pi}\mathrm{AVaR}_{\alpha}^{\pi}(C^{\infty}|X_{0}=x)= & \inf_{\pi\in\Pi}\inf_{s\in\mathbb{R}}\left\{ s+\frac{1}{1-\alpha}\mathbb{E}_{x}^{\pi}[(C^{\infty}-s)^{+}]\right\} \\
= & \inf_{s\in\mathbb{R}}\inf_{\pi\in\Pi}\left\{ s+\frac{1}{1-\alpha}\mathbb{E}_{x}^{\pi}[(C^{\infty}-s)^{+}]\right\} \\
= & \inf_{s\in\mathbb{R}}\left\{ s+\frac{1}{1-\alpha}\inf_{\pi\in\Pi}\mathbb{E}_{x}^{\pi}[(C^{\infty}-s)^{+}]\right\} 
\end{align*}
Based on this representation, we investigate the inner optimization problem for finite time $N$
as in \cite{key-4}. Let $n = 0,1,2,...,N$. We define
\begin{align}
w_{N\pi}(x,s) & := \mathbb{E}_x^{\pi}[(C^N - s)^+],~~x\in X,\quad s\in \mathbb{R},\pi \in \Pi,\\
w_{N}(x,s)    & := \inf_{\pi \in \Pi}w_{N\pi}(x,s),~~x\in X,\quad s\in \mathbb{R},
\end{align}
We work with the Markov Decision Model with a 2-dimensional state space $\tilde{X} \triangleq X \times \mathbb{R}$. The second component of the state $(x_n,s_n) \in \tilde{X}$ gives the relevant information of the history of the process. We take that there is no running cost and we assume that the terminal cost function is given by $V_{-1\pi}(x,s) := V_{-1}(x,s):= s^{-}$. We take the decision rules $f_n:\tilde{X} \rightarrow A$ such that $f_n(x,s) \in A_n(x)$ and denote by $\Pi^M$ the set of Markov policies $\pi = (f_0,f_1,...,)$, where $f_n$ are decision rules. Here, by Markov policy, we mean that the decision at time $n$ depends only on the current state $x$ and as well as  on the global variable $s$ as to be seen in the proof below.  We denote for
\begin{equation}
v \in \mathbb{M}(\tilde{X}) := \{ v : \tilde{X} \rightarrow \mathbb{R}_+: \mathrm{ measurable } \}
\end{equation} the operators:
\begin{equation}
Lv(x,s,a) := \int v (x',s-c)\mathbb{Q}(dx' | x,a ), ~~(x,s) \in \tilde{X}, a \in A_n(x)
\end{equation}
and 
\begin{equation*}
T_f v(x,s) := \int v (x',s-c)\mathbb{Q}(dx' | x,f(x,s) ), ~~(x,s) \in \tilde{X}
\end{equation*}
The minimal cost operator of the Markov Decision Model is given by 
\begin{equation}
Tv(x,s) = \inf_{a \in A_n(x)} Lv(x,s,a).
\end{equation}
For a policy $\pi = (f_0, f_1, f_2,...) \in \Pi^M$. We denote by $\vec{\pi} = (f_1,f_2,...)$ the shifted policy. We define for $\pi \in \Pi^M$ and $n = -1,0,1,...,N$:
\begin{align*}
V_{n+1,\pi} &:= T_{f_0}V_{n\pi},
\\
V_{n+1} &:= \inf_{\pi}V_{n+1\pi} 
\\
&= TV_n.
\end{align*}
A decision rule $f_n^{*}$ with the property that $V_n = T_{f^{*}_n}V_{n-1}$ is called the minimizer of $V_n$. We have \textit{Markovian} policies $\Pi^M \subset \Pi$ in the following sense: Given the global variable $s$, for every $\sigma = (f_0,f_1,...) \in \Pi^M$ we find a policy $\pi = (g_0,g_1,...) \in \Pi$ such that \begin{align*}
g_0 (x_0) &:= f_0(x_0,s) \\
g_1 (x_0,a_0, x_1) &:= f_1(x_1,s-c_0) \\
\vdots &:= \vdots
\end{align*} 
We remark here that a \textit{Markovian} policy $\sigma = (f_0,f_1,...,) \in \Pi^M$ also depends on the history of the process but not on the whole information. The necessary information at time $n$ of the history $h_n = (x_0,a_0,x_1,...,a_{n-1},x_n)$ are the state $x_n$ and the necessary information  $s_n \triangleq s_0 - c_0 -c_1 -...-c_{n-1}$. This dependence of the past and the optimality of the Markovian policy is shown in the following theorem.
\begin{Theorem}\cite{key-4}
\label{augment_info} For a given policy $\sigma$, the
only necessary information at time $n$ of the history $h_{n}=(x_{0},a_{0},x_{1,}...,a_{n-1},x_{n})$
are the followings
\begin{itemize}
\item the state $x_{n}$ 
\item the value $s_n =  s-c_0-c_1-...-c_{n-1}$ for $n=1,2,...,N$. 
\end{itemize}
Moreover, it holds for $n=0,1,...,N$ that 
\begin{itemize}
\item $w_{n\sigma} = V_{n\sigma}$ for $\sigma \in \Pi^M$.
\item $w_n = V_n$
\end{itemize}
If there exist minimizers $f^*_n$ of $V_n$ on all stages, then the Markov policy $\sigma^* = (f^*_0,...,f^*_N)$ is optimal for the problem
\begin{equation}
\inf_{\pi \in \Pi} \mathbb{E}_x^\pi [( C^N - s )^+]
\end{equation}
\end{Theorem}
\begin{proof}
For $n=0$, we obtain
\begin{align*}
V_{0\sigma}(x,s) &= T_{f_0}V_{-1}(x,s)\\
                 &= \int V_{-1}(x',s-c)\mathbb{Q}(dx'  | x,f_0(x,s)) \\
                 &= \int(s-c)^{-}\mathbb{Q}(dx'  | x,f_0(x,s))\\
                 &= \int(c-s)^{+}\mathbb{Q}(dx'  | x,f_0(x,s))\\
                 &= \mathbb{E}_x^\pi[ ( C_0 - s )^+ ] = w_{0\sigma}(x,s)
\end{align*}
Next by induction argument 
\begin{align*}
V_{n+1\sigma}(x,s) &= T_{f_0}V_{n\vec{\sigma}}(x,s) \\
              &= \int V_{n\vec{\sigma}}(x',s-c) \mathbb{Q}( dx' | x,f_0(x,s) ) \\
              &= \int \mathbb{E}_{x'}^{\vec{\sigma}} [ (C^n - (s-c))^+ ] \mathbb{Q}( dx' | x,f_0(x,s) )  \\
              &= \int \mathbb{E}_{x'}^{\vec{\sigma}} [ (c + C^n - s)^+ ] \mathbb{Q}( dx' | x,f_0(x,s) )  \\
              &= \mathbb{E}_x^\sigma [ C^{n+1} - s ]= w_{n+1\sigma}(x,s)
\end{align*}
We note that the history of the Markov Decision Process $\tilde{h}_n = (x_0,s_0,a_0,x_1,s_1,a_1,...,x_n,s_n)$ contains history $h_n = (x_0,a_0,x_1,a_1,...,x_n)$. We denote by $\tilde{\Pi}$ the history dependent policies of the Markov Decision Process. By (\cite{key-5}, Theorem 2.2.3), we get 
\begin{equation*}
\inf_{\sigma \in \Pi^M} V_{n\sigma}(x,s) = \inf_{\tilde{\pi} \in \tilde{\Pi}} V_{n\tilde{\pi}}(x,s).
\end{equation*}
Hence, we obtain 
\begin{equation*}
\inf_{\sigma \in \Pi^M} w_{n\sigma} \geq \inf_{\pi \in \Pi} w_{n\pi} \geq \inf_{\tilde{\pi} \in \tilde{\Pi}} = \inf_{\sigma \in \Pi^M}V_{n\sigma} = \inf_{\sigma \in \Pi^M} w_{n\sigma}
\end{equation*}
We conclude the proof.
\end{proof}
\begin{Theorem} \cite{key-4} \label{bauerle_finite}
Under the conditions of the Assumptions \ref{assumptions}, there exists an optimal Markov policy, in the sense introduced above, $\sigma^* \in \Pi$ for any finite horizon $N \in \mathbb{N}_0$ with 
\begin{equation}
\inf_{\pi \in \Pi} \mathbb{E}_x^\pi [(C^N - s)^+] = \mathbb{E}_x^{\sigma^*} [(C^N - s)^+] 
\end{equation}
\end{Theorem}
Now we are ready to state our main result.
\begin{Theorem} Under Assumptions \ref{assumptions}, there exists an optimal Markov policy $\pi^*$ for the infinite horizon problem (\ref{main_problem}).
\end{Theorem}
\begin{proof}
For the policy $\pi \in \Pi$ stated in the Assumption \ref{assumptions}, we have 
\begin{align}
w_{\infty,\pi} &= \mathbb{E}_x^\pi[ (C^\infty -s)^+ ]
\nonumber \\
&= \mathbb{E}_x^\pi[ (C^n + \sum_{k=n+1}^\infty C_k -s)^+ ]
\nonumber \\
&\leq E_x^\pi[ (C^n -s)^+] +E_x^\pi[ \sum_{k=n+1}^\infty C_k],
\nonumber \\
&\leq E_x^\pi[ (C^n -s)^+] +M(n),
\end{align}
where $M(n) \rightarrow 0$ as $n \rightarrow \infty$ due to the Assumption \ref{assumptions}. Taking the infimum over all $\pi \in \Pi$ we get 
\begin{equation}
w_{\infty}(x,s) \leq w_n + M(n)
\end{equation}
Hence we get 
\begin{equation}
w_n \leq w_{\infty}(x,s) \leq w_n + M(n)
\end{equation}
Letting $n \rightarrow \infty$, we get 
\begin{equation}
\lim_{n \rightarrow \infty} w_n = w_\infty
\end{equation}
Moreover, by Theorem \ref{bauerle_finite}, there exists $\pi^* =\{ f_n\}_{n=0}^N \in \Pi$ such that $V_\pi^N(x) = V^*_{0,N}(x)$ and we also have by the assumption that $V_\pi^N(x)$ is l.s.c. By the nonnegativity of the cost functions $c_n \geq 0$, we have that $N \rightarrow V^*_{0,N}(x)$ is nondecreasing and $V^*_{0,N}(x) \leq V^*_{0,\infty}(x)$ for all $x \in \mathbf{X}$. Denote 
\begin{equation}
u (x) := \sup_{N>0}V^*_{0,N}(x).
\end{equation}
Then $u(x)$ being the supremum of l.s.c. functions is l.s.c. as well. Letting $N\rightarrow \infty$, we have $u(x) \leq V^*_{0,\infty}(x)$. Hence $V^*_{0,\infty}(x)$ is l.s.c. as well. We state that the optimal policies are stationary via an induction argument as in Theorem \ref{bauerle_finite} and by Theorem 4.2.3 in \cite{key-19}, and hence conclude the proof. 
\end{proof}
\begin{Remark}
We recall that our optimization problem is 
\begin{equation}
\inf_{\pi \in \Pi}\mathrm{AVaR}_{\alpha}^\pi(\sum_{n=0}^{\infty} c(x_n,a_n))\label{eq:finite},
\end{equation}
which is equivalent to 
\begin{equation}
\inf_{\pi \in \Pi}\mathrm{AVaR}_{\alpha}^\pi(\sum_{n=0}^{\infty}c(x_n,a_n)) = \inf_{s\in\mathbb{R}}\left\{ s+\frac{1}{1-\alpha}\inf_{\pi\in\Pi}\mathbb{E}_{x}^{\pi}[(C^{\infty}-s)^{+}]\right\} 
\end{equation}
Hence, we fix the global variable a priori $s$ as 
\begin{equation}
s = \mathrm{VaR}^{\pi_0}_{\alpha}(C^\infty), 
\end{equation}
where $\mathrm{VaR}^{\pi_0}_{\alpha}(C^\infty)$ is decided using the reference probability measure $\mathbb{P}_0$. 
It is claimed in \cite{key-4} that by fixing global variable $s$, the resulting optimization problem would turn out to be over $\mathrm{AVaR}_{\beta}(C^\infty)$, where possibly $\alpha \neq \beta$, under some regularity assumptions. But, it is not clear to us, what these regularity conditions would be for that to hold and why it should be necessarily case. Since for each fixed $s$, the inner optimization problem in Equation \ref{eq:finite} has an optimal policy $\pi(s)$ depending on $s$. Hence, as in \cite{key-4}, we focus on the inner optimization problem but by fixing the global variable $s$ heuristically a priori $\mathrm{VaR}^{\pi_0}_{\alpha}(C^N)$ with respect to reference probability measure $P$ and then solve the optimization problem \textit{for each path} $\omega$ conditionally with respect to filtration $\mathcal{F}_n$ at each time $n \in \mathbb{N}_0$ namely by taking into account whether for that path $s_n \leq 0$ or $s_n > 0$. Hence, by denoting $s_n = C^n -s$, the optimization problem reduces to classical risk neutral optimization problem for that path $\omega$ whenever $s_n \leq 0$. We treat this classical case (see \cite{key-19}) in the subsection below.
\end{Remark}
\section{Solving the case when the global variable $s_n \leq 0$}
Recall that the inner optimization problem is 
\begin{align} 
V^*_0(x) &= \frac{1}{1-\alpha}\inf_{\pi \in \Pi} \mathbb{E}_x^\pi[ (C^{\infty} -s)^+].
\nonumber \\ \quad
&= \frac{1}{1-\alpha}\inf_{\pi \in \Pi} \mathbb{E}_x^\pi \bigg[ \big(\sum_{n=N+1}^\infty  c(x_n,a_n) - (s -C^N) \big)^+  \bigg]
\nonumber \\ \quad
&= \frac{1}{1-\alpha}\inf_{\pi \in \Pi} \mathbb{E}_x^\pi \bigg[ \big(\sum_{n=N+1}^\infty  c(x_n,a_n) - s_n \big)^+  \bigg] \\ \quad
&= \frac{1}{1-\alpha}\inf_{\pi \in \Pi} \mathbb{E}_x^\pi \bigg[ \mathbb{E}_x^\pi \bigg[ \big(\sum_{n=N+1}^\infty  c(x_n,a_n) - s_n \big)^+ |\mathcal{F}_n \bigg] \bigg] \\ \quad
&= \frac{1}{1-\alpha}\inf_{\pi \in \Pi} \mathbb{E}_x^\pi \bigg[ \mathbb{E}_x^\pi \bigg[ \big(\sum_{n=N+1}^\infty  c(x_n,a_n) - s_n \big)^+ |\{x_n,s_n\} \bigg] \bigg]
\end{align}
Hence, whenever $s_n \leq 0$, we have a risk neutral optimization problem in that path $\omega$. Namely, 
\begin{equation}
V_{n+1}^\pi(x) :=  \sum_{i=n+1}^{\infty} \frac{1}{1-\alpha}c_i(x_i,\pi_i) - \frac{1}{1-\alpha}s_n.
\end{equation}
Without loss of generality, with some abuse of notation, we take 
\begin{equation}
\frac{1}{1-\alpha}c_i(x_i,\pi_i) - \frac{1}{1-\alpha}s = c_i(x_i,\pi_i),
\end{equation}
for all $i \in \mathbb{N}_0$. That is to say $V_n^\pi(x_n)$ is the total cost from time $n$ onwards for that particular path $\omega$, where $n = \min \{ m \in \mathbb{N}_0 : s_m \leq 0 \}$ given the initial condition $x_n$. The corresponding minimal cost is then 
\begin{equation}
\label{optim_eqn}
V_n^*(x_n) := \inf_{\pi \in \Pi} V_n^{\pi}(x_n),
\end{equation}
We also denote that for any two integers $N > n \geq 0$
\begin{equation}
V_n^\pi(x) = V_{n,N}^\pi(x) + V_{N,\infty}^\pi(x),
\end{equation}
where 
\begin{equation}
V_{n,N}^\pi(x) := \sum_{i = n}^{N-1}c_i(x_i,a_i)
\end{equation}
is the $(N-n)$-step cost when using the policy $\pi$, starting at $x_n$ and 
\begin{equation}
V_{N,\infty}^\pi(x) := \sum_{i = N}^\infty c_i(x_i, a_i)
\end{equation}
is the tail cost from time $N$ onwards. Let 
\begin{equation}
V_{n,N}^*(x_n) := \inf_{ \pi \in \Pi } V_{n,N}^\pi(x_n)
\end{equation}
We need the following two technical lemmas. 
\begin{Lemma}
\label{critical_lemma0}
Fix an arbitrary $n \in \mathbb{N}_0$. Let $\mathbb{K}_n$ be as in assumptions, and let $u: \mathbb{K}_n \rightarrow \mathbb{R}$ be a given measurable function. Define 
\begin{equation}
u^*(x) := \inf_{a \in A_n(x)}u(x,a), \textrm{ for all } x \in X_n.
\end{equation}
\begin{itemize}
\item If $u$ is nonnegative, l.s.c. and inf-compact on $\mathbb{K}_n$, then there exists $\pi_n \in \mathbb{F}_n$ such that 
\begin{equation}
u^*(x) = u(x,\pi_n), \textrm{ for all } x \in X
\end{equation}
and $u^*$ is measurable.
\item If in addition the multifunction $x \rightarrow A_n(x)$ satisfies the Assumption \ref{assumptions}, then $u^*$ is l.s.c.
\end{itemize}
\end{Lemma}
\begin{proof}
See \cite{key-35}.
\end{proof}
\begin{Lemma}
\label{critical_lemma}
For every $N > n \geq 0$, let $w_n$ and $w_{n,N}$ be functions on $\mathbb{K}_n$, which are nonnegative,l.s.c. and inf-compact on $\mathbb{K}_n$. If $w_{n,N}\uparrow w_n$ as $N \rightarrow \infty$, then 
\begin{equation}
\lim_{N \rightarrow \infty } \min_{a \in A_n(x)} w_{n,N}(x,a) = \min_{a \in A_n(x)} w_n(x,a)
\end{equation}
for all $x \in X$. 
\end{Lemma}
\begin{proof}
See \cite{key-19} page 47.
\end{proof}
For $n \in \mathbb{N}_0$ we denote
\begin{align} 
V^*_{n,N}(x) &:= \inf_{\pi \in \Pi} \int \big( \sum_{i=n}^N c(x_n, a_n) -s  )^+\mathbb{Q}(dx'|x,f_0(x,s) \big)
\\
V^*_{n}(x) &:= \inf_{\pi \in \Pi} \int \big( \sum_{i=n}^\infty c(x_n, a_n) -s  )^+\mathbb{Q}(dx'|x,f_0(x,s) \big)
\end{align}
\begin{definition}
A sequence of functions $u_n: X_n \rightarrow \mathbb{R}$ is called a solution to the optimality equations if
\begin{equation}
u_n(x) = \inf_{a \in A_n(x)} \{ c_n(x,a) + \mathbb{E}[ u_{n+1}[F_n(x,a,\xi_n)] ] \},
\end{equation} where
\begin{equation}
\mathbb{E}[ u_{n+1}[F_n(x,a,\xi_n)] ] = \int_{S_n} u_{n+1}[F_n(x,a,s)]\mu_n(ds).
\end{equation}
\end{definition}
First, we introduce the following notations for simplicity. Let $L_n(X)$  be the family of l.s.c. non-negative functions on $X$. Moreover, denote 
\begin{equation}
\label{min_eqn}
P_nu(x) := \min_{a \in A_n(x)} \{ c_n(x,a)  + \mathbb{E}[u_{n+1}[F_n(x,a,\xi_n)]\},
\end{equation}
for all $x \in X$.

\begin{Lemma}
Using the Assumption \ref{assumptions}, then 
\begin{itemize}
\item $P_n$ maps $L_{n+1}(X)$ into $L_n(X)$.
\item For every $u \in L_{n+1}(X)$, there exists $a^*_n \in \mathbb{F}_n$ such that $a_n \in A_n(x)$ attains the minimum in (\ref{min_eqn}), i.e. 
\begin{equation}
\label{sm_eq}
P_nu(x) :=  \{ c_n(x,a_n)  + \mathbb{E}[u_{n+1}[F_n(x,a_n,\xi_n)]\},
\end{equation}
\end{itemize}
\end{Lemma}
\begin{proof}
Let $u \in L_{n+1}(X)$. Then by assumptions we have that the function 
\begin{equation}
(x,a) \rightarrow c_n(x,a) + \mathbb{E}[u_{n+1}[F_n(x,a_n,\xi_n)]
\end{equation}
is non-negative and l.s.c. and by Lemma \ref{critical_lemma0}, there exists $\pi_n \in \mathbb{F}_n$ that satisfies Equation \ref{sm_eq} and $P_n u$ is l.s.c. So we conclude the proof.
\end{proof}
By dynamic programming principle, we express the  optimality equations 38 as 
\begin{equation}
V_m^* = P_mV^*_{m+1},
\end{equation}
for all $m \geq n$. We continue with the following lemma. 
\begin{Lemma} Using the Assumption \ref{assumptions}, consider a sequence $\{ u_m \}$ of functions $u_m \in L_m(X)$ for $m \in \mathbb{N}_0$, then the following is true. If $u_n \geq P_n u_{n+1}$ for all $m \geq n$, then $u_m \geq V_m^*$ for all $m \geq n$.
\end{Lemma}
\begin{proof}
By previous lemma, there exists a policy $\pi = \{\pi_m\}_{m \geq n}$ such that for all $m \geq n$
\begin{equation}
u_m(x) \geq c_m(x,\pi_m) + u_{m+1}(x_{m+1}^\pi).
\end{equation}
By iterating, we have
\begin{equation}
u_m(x) \geq \sum_{i = m}^{N-1}c_i(x_i^\pi, \pi_i) + u_{m+N}(x_{m+N}^\pi),
\end{equation}
Hence we have 
\begin{equation}
u_m(x) \geq V_{m,N}(x,\pi),
\end{equation}
for all $N > 0$. By letting $N \rightarrow \infty$, we have $u_m(x) \geq V_m(x,\pi)$ and so $u_m \geq V_m^*$. Hence, we conclude the proof.  
\end{proof}
\begin{Theorem}
\label{val_iter}
Suppose that assumptions hold, then for every $m \geq n$ and $x \in X$,
\begin{equation}
V_{n,N}^*(x) \uparrow V_n^*(x),
\end{equation}
as $N \rightarrow \infty$ and $V_n^*$ is l.s.c. 
\end{Theorem}
\begin{proof}
We justify the statement by appealing to dynamic programming algorithm, we have $J_N(x) := 0$ for all $x \in X_N$, and by going backwards for $t=N-1,N-2,...,n$, and let
\begin{equation}
\label{eqq}
J_t(x) := \inf_{a \in A_t(x)}\{ c_t(x,a) + J_{t+1}[F_t(x,a,\xi)]\}.
\end{equation}
By backward iteration, for $t = N-1,...,n$, there exists $\pi_t \in \mathbb{F}_m$ such that $\pi_m(x) \in A_m(x)$ attains the minimum in the Equation (\ref{eqq}), and $\{ \pi_{N-1}, \pi_{N-2},...,\pi_n \}$ is an optimal policy. Moreover, $J_n$ is the optimal cost for 
\begin{equation}
J_n(x) := V^*_{n,N}(x_n),
\end{equation}
Hence, we have 
\begin{equation}
V_{n,N}^*(x) = \min_{a \in A_n(x)} \{ c_n(x,a) + V_{n+1,N}^*[F_n(x,a,\xi)] \}.
\end{equation}
Denoting $u(x) = \sup_{N>n} V^*_{n,N}(x)$, we have $u(x)$ is l.s.c. By Lemma \ref{critical_lemma}, we have 
\begin{equation}
V_n^*(x) = \min_{a \in A_n(x)} \{ c_n(x,a) + V_{n+1}^*[F_n(x,a,\xi)] \}.
\end{equation}Moreover, cost functions $c_n(x,a)$ being nonnegative, we have $u(x) \leq V_n^*(x)$. But by definition, we have $V_n^*(x) \leq u(x)$. Hence, we conclude the proof.
\end{proof} 

Intuitively, the theorem means that whenever $s_n \leq 0$ we have the risk neutral control problem where the policy is Markovian in the usual sense and hence whenever $s_n \leq 0$ we can solve the sub-problem after time $n$ using the classical dynamic programming principle. We treat the classical LQ-problem using risk sensitive AVaR operator to illustrate our results below and give a heuristic algorithm that specifies the decision rule at each time episode $n$ based on our results above.

\subsection{A Toolbox Example}
We solve the classical linear system with a quadratic one-stage cost problem with AVaR Criteria. Suppose we take $X = \mathbb{R}$ with a linear system 
\begin{equation}
x_{n+1} = x_n + a_n + Z_n,
\end{equation}
with $x_0 = 0$, $Z_n$ is i.i.d. standard normal i.e. $Z_n \sim \mathcal{N}(0,1)$. We take one stage cost functions as $c(x_n,a_n) = x_n^2 + a_n^2$ for $n = 0,1,...,N-1$. We also assume that the control constraint sets $A_n(x)$ with $x \in X$ are all equal to $A_n = \mathbb{R}$. Thus, under the above assumptions , we wish to find a policy that minimizes the performance criterion 
\begin{equation}
J(\pi,x) := \textrm{AVaR}_{\alpha}^\pi\bigg( \sum_{n=0}^{N-1} (x_n^2 + a_n^2) \bigg),
\end{equation}
It is well known that in risk neutral case using dynamic programming, the optimal policy $\pi^* = \{ f_0,...,f_{n-1} \}$ and the value function $J_n$ satisfy the following dyanimcs 
\begin{align*}
&f_{N-1}(x) = 0
\\
&f_n(x) = -(1 + K_{n+1})^{-1}K_{n+1}
\\
&K_N = 0
\\
&K_n = \bigg[ 1 - (1+ K_{n+1})^{-1}K_{n+1} \bigg]K_{n+1} + 1, \textrm{ for } n = 0,...,N-1
\\
&J_n(x) = K_nx^2 + \sum_{i = n+1}^{N-1}K_i,  \textrm{ for }n=0,...,N-1
\end{align*}
(see e.g. \cite{key-19}). In risk sensitive case, we proceed as follows.
We take $a_n = 0$ for $n = 0,...,N-1$, i.e. $\pi_0 = \{0,0,....0\}$ and let 
\begin{align*} 
s &:= \textrm{VaR}_{\alpha}(\sum_{n=0}^{N-1} c(x_n,a_n)) 
\\
&:= \inf \bigg\{ x \in \mathbb{R}: \mathbb{P}\bigg( \sum_{n=0}^{N-1} X_n^2 \leq x \bigg) \geq \alpha \bigg\}.
\end{align*} Then we check the global variable $s$. If $s \leq 0$, then we appeal to the risk neutral case and find the optimal policy accordingly. If $s > 0$, then we choose $a_0 = 0$, this makes the cost at time 0 minimal by definition. According to the output we go to time $n = 1$ and update our state to $x_1$ and repeat the procedure for each time episode $n = 0, ...,N-1$. We give the pesudocode of this algorithm below.
\begin{algorithm}
\caption{LQ Problem with AVaR algorithm}
\begin{algorithmic}[1]
\Procedure{LQ-AVaR Algorithm}{}
\State $s = \textrm{VaR}_{\alpha}^{\pi_0}(\sum_{n=0}^{N-1} X_n^2)$
\For{each $n \in N-1$ }
\If{ $s \leq 0$} 
\State apply Dynamic Programming from state $x_n$ at time $n$ onwards 
\Else 
\State Choose $a_n = 0$
\State Update $s = s - x_n^2$
\State Update $x_{n+1} = x_n + a_n + \xi_n(\omega)$
\EndIf
\EndFor
\EndProcedure
\end{algorithmic}
\end{algorithm}
\newpage

\section*{Acknowledgement}
We would like to thank Sergey Lototsky and Jianfeng Zhang  for  many  useful  comments  and  discussions.


\begin{thebibliography}{99}
\footnotesize

\bibitem{key-9}
{\sc Acciaio, B., Penner, I.} (2011). {\em Dynamic convex risk measures.}, In G. Di Nunno and B. Öksendal (Eds.), Advanced Mathematical Methods
for Finance, Springer, 1-34.

\bibitem{key-1}
{\sc Artzner, P., Delbaen, F., Eber, J.M., Heath, D.} (1999). {\em Coherent
measures of risk}, Math. Finance 9, 203-228.

\bibitem{key-2}
{\sc Aubin,J.-P., Frankowska, H.} (1978). {\em Set-Valued Analysis} Birkhauser,Boston,
1990.

\bibitem{key-4}
{\sc Bauerle, N., Ott J.} (2011). {\em Markov Decision Processes with
Average-Value-at-Risk Criteria}, Mathematical Methods of Operations
Research, 74, 361-379.

\bibitem{key-5}
{\sc Bauerle, N. , Rieder, U.} (2011). {\em Markov Decision Processes
with applications to finance}, Springer.

\bibitem{ref3}
{\sc Bellman, R.}
(1952).  {\em On the theory of dynamic programming} Proc. Natl. Acad. Sci 38, 716.

\bibitem{key-6}
{\sc Bertsekas, D., Shreve, S.E.} (1978). {\em Stochastic Optimal Control.
The Discrete Time Case}, Math. Program. Ser. B 125:235-261.

\bibitem{key-200}
{\sc Chung,K.J.,Sobel,M.J.} (1987). {\em Discounted MDPs: distribution functions and exponential utility maximization} SIAM J. Control Optimization., 25, 49-62.

\bibitem{key-12}
{\sc Ekeland, I., Temam, R.} (1974). {\em R. Convex Analysis and Variational
Problems}, Dunnod.

\bibitem{key-21}
{\sc Fleming,W., Sheu,S.} (1999). {\em Optimal long term growth rate of expected utility of wealth} Ann. Appl. Prob.,9. 871-903. 


\bibitem{key-17}
{\sc Filipovic, D. and Svindland, G.} (2012). {\em The canonical model space for law-invariant convex
risk measures is L1 }, Mathematical Finance 22(3), 585-589.

\bibitem{key-16}
{\sc Guo, X., Hernandez-Lerma, O.} (2012). {\em Nonstationary discrete-time deterministic
and stochastic control systems with infinite horizon}, International
Journal of Control, vol. 83, pp 1751-1757.

\bibitem{key-19}
{\sc  Hernandez-Lerma,O., Lasserre, J.B.} (1996). {\em Discrete-time Markov
Control Processes. Basic Optimality Criteria.}, Springer,New York.


\bibitem{key-18}
{\sc  Kupper, M., Schachermayer, W.} (2009). {\em Representation results
for law invariant time consistent functions},Mathematics and Financial
Economics 189-210.

\bibitem{key-8}
{\sc Rockafellar, R.T , Uryasev, S.} (2002). {\em Conditional-Value-at-Risk
for general loss distributions}, Journal of Banking and Finance 26,
1443-1471.

\bibitem{key-10}
{\sc Rockafellar, R.T., Wets, R.J.-B.} (1998). {\em Variational Analysis.}, Springer, Berlin.

\bibitem{key-11}
{\sc Ruschendorf, L., Kaina, M.} (2009). {\em On convex risk measures
on Lp-spaces}, Mathematical Methods in Operations Research, 475-495.

\bibitem{key-3}
{\sc Ruszcynski, A.} (1999). {\em Risk-averse dynamic programming for
Markov decision processes}, Math. Program. Ser. B 125:235-261.




\bibitem{key-14}
{\sc Ruszczynski, A. and Shapiro, A.} (2006). {\em Optimization
of convex risk functions}, Mathematics of Operations
Research, vol. 31, pp. 433-452.

\bibitem{key-15}
{\sc Shapiro, A.} (2012). {\em Time consistency of dynamic risk measures}, 
Operations Research Letters, vol. 40, pp. 436-439.

\bibitem{key-20}
{\sc  Xin, L., Shapiro, A.} (2009). {\em Bounds for
nested law invariant coherent risk measures}, Operations
Research Letters, vol. 40, pp. 431-435.

\bibitem{key-30}
{\sc Shapiro, A.} (2015).  {\em Rectangular sets of probability measures}, preprint.

\bibitem{key-31}
{\sc Epstein, L. G. and Schneider, M.} (2003).  {\em Recursive multiple-priors}, Journal of Economic Theory, 113, 1-31.

\bibitem{key-33}
{\sc Iyengar, G.N.} (2005).  {\em Robust Dynamic Programming}, Mathematics of Operations Research, 30, 257-280.

\bibitem{key-35} 
{\sc Rieder, U.} (1978). {\em Measurable Selection Theorems for Optimisation Problems}, Manuscripta Mathematica, 24, 115-131.
\end{thebibliography}
\end{document}